\documentclass[11pt,reqno]{amsart}
\usepackage{amssymb}
\usepackage{amsmath}
 \usepackage[usenames,dvipsnames]{color}
\usepackage[kerning=true]{microtype}
\usepackage[page]{appendix}
\usepackage[all,cmtip]{xy}
 \usepackage{pgf,tikz}
 \usetikzlibrary{decorations.markings}
\usetikzlibrary{arrows,shapes,positioning,calc,shadows}
\tikzstyle arrowstyle=[scale=1]
\tikzstyle directed=[postaction={decorate,decoration={markings,mark=at position .5 with {\arrow[arrowstyle]{stealth}}}}]
\usepackage{fullpage}
\usepackage{amsthm}
\usepackage{bm}
\usepackage[T1]{fontenc}
\usepackage{graphicx}
\usepackage[applemac]{inputenc}
\usepackage[mathcal]{euscript}
\usepackage[french]{todonotes}
\usepackage[colorlinks=true, citecolor=cyan,urlcolor=blue,linktocpage=true,pagebackref]{hyperref} 

\definecolor{midgreen}{rgb}{0.,0.7,0.} 
\definecolor{green}{rgb}{0.,0.6,0.} 
\def\rouge{\textcolor{red}}
\def\bleu{\textcolor{blue}}

\def\vert{\textcolor{green}}

\newcommand{\facteur}{.95}
\newcommand{\dx}{0.01}
\newcommand{\dy}{-.99}
\newcommand{\jambe}[2]{\path [fill=yellow] (\dx+#1,\dy+#2) rectangle ($(\dx+#1,\dy+#2)+\facteur*(1,1)$);}
\newcommand{\jambepale}[2]{\path [fill=yellow!15] (\dx+#1,\dy+#2) rectangle ($(\dx+#1,\dy+#2)+\facteur*(1,1)$);}

\newcommand{\jambeverte}[2]{\path [fill=midgreen] (\dx+#1,\dy+#2) rectangle ($(\dx+#1,\dy+#2)+(1,1)$);}
\newcommand{\caseverteetiquetee}[3]{\path [fill=green] (\dx+#1,\dy+#2) rectangle ($(\dx+#1,\dy+#2)+(1,1)$);
                                                               \node[color=yellow] at (#1+.55,#2-.45){\Large \bf{#3}};}
\newcommand{\carrejaune}[2]{\filldraw[draw=black,fill=yellow,line width=.15mm] (#1-.5,#2-.5) rectangle (#1+.5,#2+.5);}
\newcommand{\carrevert}[2]{\filldraw[draw=black,fill=midgreen,line width=.15mm] (#1-.5,#2-.5) rectangle (#1+.5,#2+.5);}
\newcommand{\carrebleu}[2]{\filldraw[draw=black,line width=.15mm, fill=blue, inner sep=1pt] (#1-.5,#2-.5) rectangle (#1+.5,#2+.5);}
\newcommand{\cellrow}[2]{\foreach \x in {0,...,#2} {\jambe{\x}{#1} ;}}
\newcommand{\cellrowpale}[2]{\foreach \x in {0,...,#2} {\jambepale{\x}{#1} ;}}

\newcommand{\crochet}[4]{\coordinate (A) at (#1,#2); 
					  \coordinate (B) at ($(A)+(#3,0)$); 
					  \coordinate (C) at ($(B)-(0,#4)$); 
                                            \draw[red,ultra thick,line join=round,line cap=round] (A) -- (B) -- (C) ;}

\DeclareMathOperator{\area}{area}

\newcommand{\A}{\mathcal{A}}
\newcommand{\coeff}{\bm{c}}
\DeclareMathOperator{\card}{card}
\newcommand{\coeffprime}{\bm{d}}
\newcommand{\charac}{\raise 2pt\hbox{\large$\chi$}}
\DeclareMathOperator{\dinv}{dinv}

\newcommand{\eperp}[1]{(e_{#1}^\perp\otimes \Id)}
\newcommand{\E}{\mathcal{E}}
\newcommand{\F}{\mathcal{T}}
\DeclareMathOperator{\GL}{GL}
\DeclareMathOperator{\Id}{Id}
\newcommand{\M}{\mathcal{M}}
\newcommand{\Mnk}[2]{\M_{#1}^{\langle #2\rangle}}
\newcommand{\N}{\mathbb{N}}

\newcommand{\Rational}{\mathbb{Q}}

\newcommand{\scalar}[2]{{\langle#1,#2 \rangle}}
\DeclareMathOperator{\SSYT}{SSYT}
\renewcommand{\S}{\mathbb{S}}
\newcommand{\Super}{\mathcal{S}}

\newcommand{\Snk}[2]{\Super_{#1}^{\langle #2\rangle}}

\newcommand{\auteur}[1]{{\sc #1}}
\newcommand{\titreref}[1]{{\em #1}}
\newcommand{\vol}[1]{{\bf #1}}
\newcommand{\MathReview}[1]{}

\newcommand{\pref}[1]{{\rm (\ref{#1})}}

\newcommand{\define}[1]{\bleu{\bf{#1}}}

\newtheorem{theorem}{\bleu{Theorem}}
\newtheorem{lemma}{\bleu{Lemma}}
\newtheorem{conjecture}{\bleu{Conjecture}}

\newtheorem{prop}{\bleu{Proposition}}

\numberwithin{equation}{section}
\numberwithin{lemma}{section}
\numberwithin{rmk}{section}
\numberwithin{cor}{section}

\title[Multi]{$(\GL_k\times\S_n)$-modules and
  Nabla of Hook-Indexed Schur functions}
\author{F.~Bergeron}
\address{\href{http://bergeron.math.uqam.ca}{Département de Mathématiques, Lacim, UQAM.}}
\thanks{During this research, the author was supported by a NSERC grant.}
  \email{\href{mailto:bergeron.francois@uqam.ca}{bergeron.francois@uqam.ca}}
  \date{\bleu{\bf \today}}
\urladdr{bergeron.math.uqam.ca}

\begin{document}


\begin{abstract} 
The aim of this paper is to describe structural properties of spaces of diagonal rectangular harmonic polynomials in several sets (say $k$) of $n$ variables, both as $\GL_k$-modules and $\S_n$-modules. We construct explicit such modules associated to any hook shape partitions. For the two sets of variables case, we conjecture that the associated graded Frobenius characteristic corresponds to the effect of the operator Nabla on the corresponding hook-indexed Schur function, up to a usual renormalization. We prove identities  that give indirect support to this conjecture, and show that its restriction to one set of variables holds. We further give indications on how the several sets context gives a better understanding of questions regarding the structures of these modules and the links between them.
\end{abstract}
\maketitle

\footskip=30pt

  
\parindent=20pt \parskip=8pt


\section{\bleu{Introduction}}
Our aim in this paper is to describe, for all hook-shape partitions $\rho$, new $(\GL_k\times\S_n)$-modules of $k$-variate diagonal harmonic polynomials, here denoted by $\mathcal{S}_{\rho}^{\langle k\rangle}$, whose $\N^k$-graded Frobenius characteristic specializes to $\nabla(\widehat{s}_{\rho})$, when  one sets $k=2$. Here $\widehat{s}_{\rho}$ stands for the normalized Schur function\footnote{We essentially use Macdonald's notations (see~\cite{macdonald}), but with French conventions.}
\begin{equation}
    \widehat{s}_{\rho}:=\big({\textstyle \frac{-1}{qt}}\big)^{\iota_\rho} s_\rho,\qquad {\rm with}\qquad \iota_\rho:=\sum_{\rho_i>i} \rho_i-i.
\end{equation} 
We also recall that $\nabla$ is the operator (introduced in~\cite{ScienceFiction}) on symmetric polynomials (with coefficients in $\Rational(q,t)$) for which the ``combinatorial Macdonald polynomials'' $\widetilde{H}_\lambda(q,t;\bm{z})$ are joint eigenfunctions, with eigenvalue $T_\lambda=T_\lambda(q,t):=\prod_{(i,i)} q^i t^j$. Here the product is over the cells $(i,j)$ of $\lambda$. For more background on these notions, see~\cite{identity,haimanhilb}.
 
For hook shape partitions $(a+1,1^b)$, of $n=a+b+1$, we will use the \define{Frobenius's notation}  $(a\,|\,b)$
 (see Figure~\ref{Fig_hook_shape}). For example, we have
   $$(4\,|\,0)=5,\qquad (3\,|\,1)=41,\qquad (2\,|\,2)=311,\qquad (1\,|\,3)=2111,\qquad (0\,|\,4)=11111.$$
Observe that, for $\rho=(a\,|\,b)$, the value of $\iota(\rho)$ is simply equal to $a$. 
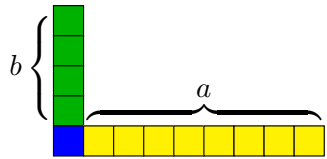
\begin{figure}[ht]
\begin{center}
 \begin{tikzpicture}[thick,scale=.4]
 \node at (-1.2,2.5) {$b\left\{\rule{0cm}{24pt}\right.$};
 \node at (4.5,1.3) {$\overbrace{\hskip3.1cm}^{\textstyle a}$};
\carrevert{0}{4}
\carrevert{0}{3}
\carrevert{0}{2}
\carrevert{0}{1}
\carrebleu{0}{0}
\carrejaune{1}{0}\carrejaune{2}{0}\carrejaune{3}{0}\carrejaune{4}{0}\carrejaune{5}{0}\carrejaune{6}{0}\carrejaune{7}{0}\carrejaune{8}{0}
\end{tikzpicture}
\end{center} \vskip-10pt
\caption{The hook shape $(a\,|\,b)$.}
\label{Fig_hook_shape} 
\end{figure}

\noindent
Our   graded modules  $\mathcal{S}_{\rho}^{\langle k\rangle}$ are associated to modules $\Mnk{\rho}{k}$. These occur in a bi-filtration of $\GL_k\times \S_n$-modules (over the field $\Rational$), with rows indexed by hooks going from $(n-1\,|\,0)$ to $(0\,|\,{n-1})$, hence in increasing number of $1$'s; columns correspond to integers $k\in\N^+$ (as numbers of ``sets'' of variables)
\begin{displaymath}
\xymatrix@-=0.4cm{
            \Mnk{(n-1\,|\,0)}{1} \ar@{^{(}->}[d]    \ar@{^{(}->}[r] 
	& \Mnk{(n-1\,|\,0)}{2} \ar@{^{(}->}[d]  \ar@{^{(}->}[r] 
	& \ar@{.}[r]
	&\ \ar@{^{(}->}[r]
	&\Mnk{(n-1\,|\,0)}{k}  \ar@{^{(}->}[d]  \ar@{^{(}.>}[rr]^{\hskip-3pt \scriptscriptstyle k\to \infty}
	&\ 
	& \bleu{\M_{(n-1\,|\,0)} }\ar@{^{(}->}[d] \\
             \Mnk{(n-2\,|\,1)}{1} \ar@{^{(}->}[d]   \ar@{^{(}->}[r]  
 	&  \Mnk{(n-2\,|\,1)}{2} \ar@{^{(}->}[d]   \ar@{^{(}->}[r] 
	& \ar@{.}[r]
	&\ \ar@{^{(}->}[r]  
	&\Mnk{(n-2\,|\,1)}{k} \ar@{^{(}->}[d]  \ar@{^{(}.>}[rr]^{\hskip-3pt \scriptscriptstyle k\to \infty}  
	&\ 
	&\bleu{ \M_{(n-2\,|\,1)}} \ar@{^{(}->}[d]  \\ 
	  \ar@{.}[d]  
	&\ar@{.}[d]
	&  
	& 
	&\ar@{.}[d]
	&
	&\ar@{.}[d] \\
 	   \ar@{^{(}->}[d] 
	&\ar@{^{(}->}[d]
	&
	&
	&\ar@{^{(}->}[d]
	&
	&\ar@{^{(}->}[d] \\
	   \Mnk{(1\,|\,n-2)}{1} \ar@{^{(}->}[d]    \ar@{^{(}->}[r] 
	& \Mnk{(1\,|\,n-2)}{2} \ar@{^{(}->}[d]  \ar@{^{(}->}[r] 
	& \ar@{.}[r]
	&\  \ar@{^{(}->}[r]
	&\Mnk{(1\,|\,n-2)}{k}  \ar@{^{(}->}[d]  \ar@{^{(}.>}[rr]^{\hskip-3pt \scriptscriptstyle k\to \infty}
	&\ 
	& \bleu{\M_{(1\,|\,n-2)}} \ar@{^{(}->}[d] \\
	    \Mnk{(0\,|\,n-1)}{1}   \ar@{^{(}->}[r] 
	& \Mnk{(0\,|\,n-1)}{2}  \ar@{^{(}->}[r] 
	& \ar@{.}[r]
	&\ \ar@{^{(}->}[r] 
	&\Mnk{(0\,|\,n-1)}{k}   \ar@{^{(}.>}[rr]^{\hskip-3pt \scriptscriptstyle k\to \infty}
 	&\ 
	&\bleu{ \M_{(0\,|\,n-1)} }}
\end{displaymath}
Each row stabilizes 
when $k$ becomes large enough; and we have the inductive limits 
    $$\bleu{\M_{(a\,|\,b)}}:=\lim_{k\to \infty}\Mnk{(a\,|\,b)}{k},$$ 
which are $\GL_\infty\times \S_n$-modules. It is convenient to set $\M_{(n\,|\,-1)}:=0$, and then consider the quotient modules
    $$\bleu{\Super_{(a\,|\,b)}}=\M_{(a\,|\,b)}/ \M_{(a+1,b-1)},$$
 for all hooks $(a\,|\,b)$.
Explicitly  $\M_{(a\,|\,b)}$ is the smallest module which is:
\begin{itemize}
\item[$\bullet$] closed for ``polarization'' (see definition~\pref{def_polarization} below), 
\item[$\bullet$] closed for derivation with respect to all variables except the $\theta_i$'s, and 
\item[$\bullet$] contains the determinant:
\end{itemize}
 \begin{equation}\label{def_Determinant_hook}
\bm{D}_{(a\,|\,b)}(\bm{x}):=\det \begin{pmatrix} 
     \bleu{\theta_1} &1 & x_1 &\cdots & \widehat{x_1^a} &\cdots  &  x_1^{n-1}\\
     \bleu{\theta_i}& 1 & x_2 & \cdots & \widehat{x_2^a} & \cdots&  x_2^{n-1}\\
     \vdots &\vdots  & \vdots & \ddots & \vdots & \ddots & \vdots \\
       \bleu{ \theta_n}& 1 & x_n  &\cdots& \widehat{x_n^a}  &\cdots &  x_n^{n-1}   
    \end{pmatrix};
 \end{equation}
where $\widehat{(-)}$ indicates that entries of that column are removed. Only the first column involves the variables $\theta_i$, which are said to be \define{inert} and considered to be of $0$-degree. Observe that, if we remove the first column (the $\theta$-column) and keep all others, the result is the classical Vandermonde determinant. The inclusions occurring in the columns are easily obtained if one observes that 
  $$\sum_{i=1}^n \partial x_i\, \bm{D}_{(a\,|\,b)} =\begin{cases}
   (a+1)\, \bm{D}_{(a+1\,|\,b-1)},   & \text{ if}\ 0\leq a<n-1, \\[4pt]
   0,   & \text{otherwise}.
\end{cases}$$
By construction, $\M_\rho$ is a (multi-)homogenous sub-module of the $\N^\infty$-graded ring $\bm{R}:=\Rational[\bm{X}]$ of polynomials in the set of variables $\bm{X}$, consisting of a denumerable number of sets of $n$-variables\footnote{The modules $\Mnk{\rho}{k}$ are likewise defined with variables restricted to the first $k$ rows of $\bm{X}$.}. The variables $\bm{X}$ are conveniently presented as an $\infty\times n$ matrix
   \begin{equation}
      \bm{X} = \begin{pmatrix} x_1& x_2&\ldots & x_n\\
         y_1& y_2&\ldots & y_n\\
         \vdots &\vdots &\ddots &\vdots\\
       \end{pmatrix}
   \end{equation}
 in which the first row is the ``set'' $\bm{x}$: of variables $x_i$, for $1\leq i\leq n$. The grading of a polynomial $F$ in $\bm{R}$ is the sequence of degrees $\deg(F)=(\deg_{\bm{x}}(F),\deg_{\bm{y}}(F),\ldots)$, respective to each of the rows of $\bm{X}$. The  natural commuting actions of $\GL_\infty$ and $\S_n$, on polynomials $F(\bm{X})$ in $\bm{R}$, are jointly described by setting
    \begin{equation}
       F(\bm{X})\mapsto F(g \bm{X}\sigma),\qquad{\rm for}\qquad g\in\GL_\infty,\quad {\rm and}\quad \sigma\in \S_n,
    \end{equation}
where elements of the symmetric group $\S_n$ are here considered as $n\times n$ permutation matrices. For each pair $\bm{u}=(u_i)_i$ and $\bm{v}=(v_i)_i$ of rows of $\bm{X}$, and integer $r\geq 1$, the \define{(higher) polarization operator} $E_{\bm{u}\bm{v}}^{(r)}$ is:
   \begin{equation}\label{def_polarization}
       E_{\bm{u}\bm{v}}^{(r)} :=\sum_{i=1}^n v_i\,\frac{\partial^r}{\partial u_i^r}.
   \end{equation}
 We often drop the super index ``$(r)$'' when $r=1$. In formula, we may then write the definition of  $\M_{(a\,|\,b)}$  as
 $$\M_{(a\,|\,b)}:=\Rational(\{\partial x_i\}_{x_i\in\bm{x}}; \{E_{\bm{u}\bm{v}}^{(r)} \}_{r,\bm{u},\bm{v}})\, \bm{D}_{(a\,|\,b)}(\bm{x}).$$
 The \bleu{($\N^\infty$-)graded Frobenius characteristic} of such a module $\M_\rho$, for $\rho={(a\,|\,b)}$, is the generating function of the characters of its graded components, that is: 
\begin{align}
  \M_\rho(\bm{q};\bm{z})&:=\sum_{\bm{d}\in \N^\infty} \bm{q}^{\bm{d}} \sum_{\mu\vdash n} \charac^\rho_{\bm{d}}(\mu) \frac{p_\mu(\bm{z})}{z_\mu},\nonumber\\
       &=\sum_{\mu\vdash n} \Big(\sum_{f\in \mathcal{B}_{\rho}^\mu} \bm{q}^{\deg(f)}\Big) s_\mu(\bm{z}),\label{Isotypic_Hilbert}
 \end{align}
 where  $\bm{q}^{\bm{d}}:=q_1^{d_1} q_2^{d_2} \cdots$,
for $\bm{d}=(d_1,d_2,\ldots)$, and $\chi^\rho_{\bm{d}}$ is the character of the $\bm{d}$-homogenous component of $\M_\rho$. We recall that, with respect to the Frobenius map, irreducible $\S_n$-representations are precisely encoded by Schur functions $s_\mu(\bm{z})$, with $\bm{z}=(z_i)_i$ a set of formal variables. 
The expansion in~\pref{Isotypic_Hilbert}
 corresponds to the decomposition of $\M_\rho$ into $\S_n$-isotypic components $\M_\rho^\mu$, one for each partition $\mu$ of $n$. Indeed, these $\M_\rho^\mu$'s are clearly graded, and they afford bases of homogeneous polynomials $\mathcal{B}_\rho^\mu$. Hence, considering $\bm{q}$ as a formal diagonal matrix in $\GL_\infty$, we may express this homogeneity of $F\in \mathcal{B}_\rho^\mu$  as
     $$F(\bm{q}\bm{X}) = \bm{q}^{\deg(F)}\, F(\bm{X}).$$
 The coefficients of each $s_\mu(\bm{z})$ in~\pref{Isotypic_Hilbert} may thus be considered, either as the Hilbert series of the corresponding $\S_n$-isotypic components, or as $\GL_\infty$-characters of (polynomial) representations. Recall that the characters of polynomials irreducible $\GL_\infty$-representations are also Schur functions (here in the variables $\bm{q}$). It follows that we have 
\begin{equation}
  \M_\rho(\bm{q};\bm{z})=\sum_{\mu\vdash n}\sum_\lambda a_{\lambda\mu}^\rho s_\lambda(\bm{q}) s_\mu(\bm{z}),
   \end{equation}
where each of the integer $a_{\lambda\mu}^\rho$ gives the number of copies of the $\GL_\infty$-irreducible having character $s_\lambda(\bm{q})$ in the $\S_n$-isotypic component $\M_\rho^\mu$.

In light of the discussion above, for $\rho=(a\,|\,b)$, we have
    \begin{align}
  \Super_\rho(\bm{q};\bm{z}) &:=  \M_{(a\,|\,b)}(\bm{q};\bm{z})- \M_{(a+1\,|\,b-1)}(\bm{q};\bm{z})\nonumber \\
    &=\sum_{\mu\vdash n}\sum_\lambda c_{\lambda\mu}^\rho s_\lambda(\bm{q}) s_\mu(\bm{z}),
 \end{align}
 with the \define{multiplicities} $c_{\lambda\mu}^\rho\in \N$ equal to $a_{\lambda\mu}^{(a\,|\,b)}-a_{\lambda\mu}^{(a+1\,|\,b-1)}$. It is convenient to express this in a ``tensor'' variable-free format
    \begin{equation}
       \Super_\rho= \sum_{\mu\vdash n}\sum_\lambda c_{\lambda\mu}^\rho s_\lambda\otimes s_\mu,
     \end{equation}
with the tensor product keeping track of the distinction between $\GL_\infty$-characters (left-hand side) and Frobenius of $\S_n$-irreducibles (right-hand side). The {multiplicities}  $c_{\lambda\mu}^\rho$  are only non-vanishing when the partition $\lambda$ has at most $n-1$ parts, with size bounded by $\binom{n}{2}+b$ (for $\rho=(a\,|\,b)$).
  We denote by   $\coeff_{\rho\mu}$ the \define{coefficient} of $s_\mu$ in $\Super_\rho$, and also write $\A_\rho$ for  $\coeff_{\rho,1^n}$. In formulas:
   \begin{equation}
       \coeff_{\rho\mu}=\sum_{\lambda} c_{\lambda\mu}^\rho s_\lambda,\qquad{\rm and}\qquad \A_{\rho}=\sum_{\lambda} c_{\lambda,1^n}^\rho s_\lambda.
    \end{equation}
We also consider the ``scalar product''
such that $\scalar{ f\otimes s_\nu}{s_\mu} = f$, so that $\coeff_{\rho\mu}=\scalar{\Super_\rho}{s_\mu}$. The \define{length restriction operator} $L_{\leq k}$ effect on $\Super_\rho$ is set to be:
\begin{equation}
   \rouge{L_{\leq k}}( \Super_\rho):=\sum_{\mu\vdash n}\sum_{\rouge{\ell(\lambda)\leq k}} c_{\lambda\mu}^\rho s_\lambda\otimes s_\mu.
\end{equation}
We may also write
\begin{equation}
   \rouge{L_{\leq k}}( \M_\rho)=\sum_{\rouge{0\leq i\leq k}} \Mnk{\rho}{i},
   \end{equation}
considering that each $\Mnk{\rho}{i}$ is understood to be in tensor format.


 \subsection*{Effect of \texorpdfstring{$\nabla$}{N} on Schur functions indexed by hooks.} 
To better express our main conjecture, we consider the $s_\lambda\otimes s_\mu$-expansion of $\nabla(\widehat{s}_\rho)$ (with $\rho$ a hook as above):
   \begin{align}
     \nabla(\widehat{s}_\rho) &= \sum_{\mu\vdash n} \sum_{\lambda} b_{\lambda\mu}^\rho s_\lambda \otimes s_\mu,\label{Expansion_nabla}
   \end{align}
which corresponds to the fact that $\nabla(\widehat{s}_\rho)(q,t;\bm{z})$ affords the expansion
      \begin{align*}
       \nabla(\widehat{s}_\rho)(q,t;\bm{z}) &= \sum_{\mu\vdash n} \sum_\lambda b_{\lambda\mu}^\rho s_\lambda(q,t) s_\mu(\bm{z}).
     \end{align*}
Recall that $s_\lambda(q,t)=0$ whenever the partition $\lambda$ has more than two parts. In other words, only partitions of {length}  at most two (in formula $\ell(\lambda)\leq 2$) may contribute to the internal sum in~\pref{Expansion_nabla}.

Now, let 
\begin{equation}\label{defn_staircase}
\bleu{\delta^{(n)}:=(n-1,n-2,\cdots, 2,1,0)}
\end{equation}
 be the \define{$n$-staircase} partition (see figure below). A \define{Dyck path} $\gamma$ may be identified to a partition contained in $\delta^{(n)}$ (equivalently it lies below the \define{diagonal} going from $(0,n)$ to $(n,0)$), as is illustrated in Figure~\ref{Fig_Dyck_path}.

 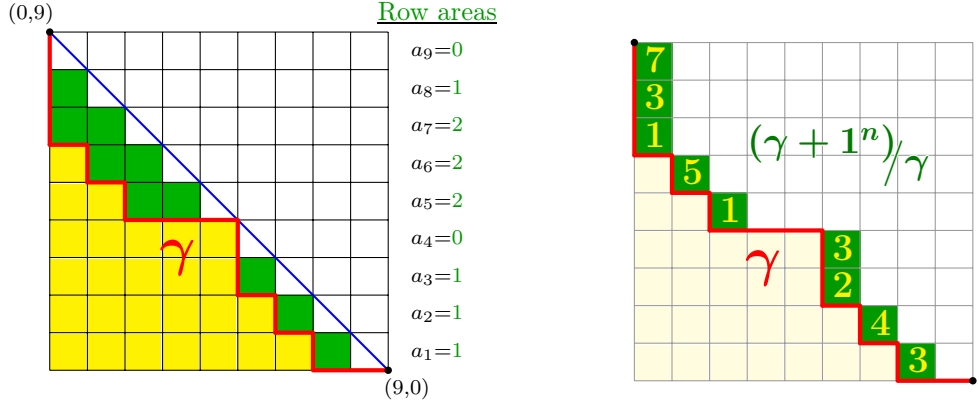
\begin{figure}[ht]
\begin{center}
 \begin{tikzpicture}[thick,scale=.5]
 \coordinate (NW) at (0,9);
 \coordinate (SE) at (9,0);
\cellrow{6}{0}
\cellrow{5}{1}
\cellrow{4}{4}
\cellrow{3}{4}
\cellrow{2}{5}
\cellrow{1}{6}
\jambeverte{0}{8}
\jambeverte{0}{7}\jambeverte{1}{7}
\jambeverte{1}{6}\jambeverte{2}{6}
\jambeverte{2}{5}\jambeverte{3}{5}
\jambeverte{5}{3}
\jambeverte{6}{2}
\jambeverte{7}{1}
 \draw[step=1.0,black,thin] (0,0) grid (9,9);
  \draw[blue] (NW) to (SE) ;
\crochet{0}{9}{0}{3}
\crochet{0}{6}{1}{1}
\crochet{1}{5}{1}{1}
\crochet{2}{4}{3}{2}
\crochet{5}{2}{1}{1}
\crochet{6}{1}{1}{1}
 \draw[red,ultra thick] (7,0) to (9,0) ;
  \filldraw[black] (0,9) circle (.07);
   \filldraw[black] (9,0) circle (.07);
\node at (-.5,9.5) {$\scriptstyle (0,9)$};
\node at (9.5,-.5) {$\scriptstyle (9,0)$};
\node at (3.4,3) {{\huge \rouge{$\bm\gamma$}}};
\node at (10.3,8.5) {$\scriptstyle {a_9}=\vert{0}$};
\node at (10.3,7.5) {$\scriptstyle {a_8}=\vert{1}$};
\node at (10.3,6.5) {$\scriptstyle {a_7}=\vert{2}$};
\node at (10.3,5.5) {$\scriptstyle {a_6}=\vert{2}$};
\node at (10.3,4.5) {$\scriptstyle {a_5}=\vert{2}$};
\node at (10.3,3.5) {$\scriptstyle {a_4}=\vert{0}$};
\node at (10.3,2.5) {$\scriptstyle {a_3}=\vert{1}$};
\node at (10.3,1.5) {$\scriptstyle {a_2}=\vert{1}$};
\node at (10.3,0.5) {$\scriptstyle {a_1}=\vert{1}$};
\node at (10.3,9.5) {\small\underline{\vert{Row areas}}};
\end{tikzpicture}
\qquad\qquad 
 \begin{tikzpicture}[thick,scale=.5]
 \cellrowpale{6}{0}
\cellrowpale{5}{1}
\cellrowpale{4}{4}
\cellrowpale{3}{4}
\cellrowpale{2}{5}
\cellrowpale{1}{6}
\caseverteetiquetee{0}{9}{7}
\caseverteetiquetee{0}{8}{3}
\caseverteetiquetee{0}{7}{1}
\caseverteetiquetee{1}{6} {5}
\caseverteetiquetee{2}{5} {1}
\caseverteetiquetee{5}{4}{3}
\caseverteetiquetee{5}{3}{2}
\caseverteetiquetee{6}{2}{4}
\caseverteetiquetee{7}{1}{3}
 \draw[step=1.0,gray!80,thin] (0,0) grid (9,9);
\crochet{0}{9}{0}{3}
\crochet{0}{6}{1}{1}
\crochet{1}{5}{1}{1}
\crochet{2}{4}{3}{2}
\crochet{5}{2}{1}{1}
\crochet{6}{1}{1}{1}
 \draw[red,ultra thick] (7,0) to (9,0) ;
  \filldraw[black] (0,9) circle (.07);
   \filldraw[black] (9,0) circle (.07);
\node at (3.2,3) {\huge{ \rouge{$\bm{\gamma}$}}};
\node[color=white] at (5.42,6.02) {\Large$\bm{(\gamma+1^n)_{\displaystyle\!\!  \big/\gamma}}$};
\node[color=white] at (5.42,5.98) {\Large$\bm{(\gamma+1^n)_{\displaystyle\!\!   \big/\gamma}}$};
\node[color=white] at (5.46,5.98) {\Large$\bm{(\gamma+1^n)_{\displaystyle \!\!  \big/\gamma}}$};
\node[color=white] at (5.46,6.02) {\Large$\bm{(\gamma+1^n)_{\displaystyle \!\!  \big/\gamma}}$};
\node[color=green!80!black] at (5.44,6.0) {\Large$\bm{(\gamma+1^n)_{\displaystyle \!\!  \big/\gamma}}$};
\node at (9.5,-.5) {$\quad$};
\end{tikzpicture}
\end{center}
\qquad  \vskip-15pt
\caption{Dyck path $\rouge{\bm{\gamma}}=765521000$, and one of it associated skew-shaped SSYT's.}
\label{Fig_Dyck_path} 
\end{figure}

\noindent For each row $\gamma_i$ of $\gamma$, one considers the  \define{row area} $a_i=\delta^{(n)}_i-a_i$. In other words, $a_i$ is the number of cells lying on the row $i$  between the Dyck path and the diagonal, and $\gamma=\delta^{(n)}-\alpha(\gamma)$, with $\alpha(\gamma):=(a_1,a_2,\ldots ,a_n)$.

It follows from 
results of~\cite{HMZ}, together with the composition shuffle theorems of~\cite{carlsson_mellit,mellit_braids}, that we have the following combinatorial formula.
\begin{prop}[Shuffle formula]
\begin{equation}\label{combinatorial_nabla}
    \nabla(\widehat{s}_{(a\,|\, b)})(q,t;\bm{z}) = \sum_{\gamma\subseteq \Gamma_a} t^{\area(\gamma)-{a}}\, \mathbb{L}_\gamma(q;\bm{z}),
\end{equation}    
where $\gamma$ runs over the set of Dyck paths contained in 
\begin{equation}\label{def_Gamma_a}
   \Gamma_a:=\delta^{(n)} - (0,\ldots,0,\underbrace{1,\ldots,1}_{a\ {\rm copies}},0);
 \end{equation}
and $\mathbb{L}_\gamma$ is the associated LLT-polynomial. 
\end{prop}
We recall that the LLT-polynomial $\mathbb{L}_\gamma(q;\bm{z})$, of a Dyck path $\gamma$, is an instance of vertical-strip LLT-polynomial (see ~\cite{Panova}, which includes a short survey of generalized LLT-polynomials). It is obtained as a weighted sum over the set $\SSYT((\gamma+1^n)/\gamma)$ of semi-standard Young tableaux\footnote{Whose shape is the set of cells siting immediately to the right a vertical step of the Dyck-path $\gamma$. See Fig.~\ref{Fig_Dyck_path}}  of skew shape $(\gamma+1^n)/\gamma$:
\begin{equation}
  \mathbb{L}_\gamma(q;\bm{z}):=\sum_{\tau \in \SSYT(\gamma+1^n)} t^{\dinv(\tau)} \bm{z}_\tau,
\end{equation}
with $\bm{z}_\tau$ equal to the product of $z_i$ over entries $i$ of $\tau$. For details of the $\dinv$-statistic for skew shape semi-standard Young tableaux, see~\cite{DinvArea}. It has been shown (see~\cite{Grojnowski,HHLRU}) that $\mathbb{L}_\gamma$ is Schur-positive.

Until now, the combinatorial description~\pref{combinatorial_nabla} of $ \nabla(\widehat{s}_{(a\,|\, b)})$  lacked a representation theory counterpart, {i.e.} a module for which it is the graded Frobenius. We now propose the following.
\begin{conjecture}[Modules]\label{Conj_Modules}
For all hook indexed shape $\rho=(a\,|\, b)$,  $\Super_\rho$ is such that 
 \begin{equation}\label{conj_module_formule}
    L_{\leq 2} (\Super_\rho)= \nabla(\widehat{s}_\rho).
 \end{equation}  
 In other words, $c_{\lambda\mu}^\rho=b_{\lambda\mu}^\rho$ for all $\rho$, $\mu$, and partitions $\lambda$ having length at most two.
\end{conjecture}
It may be worth underlying that to calculate $\Super_\rho(q,t;\bm{z})$, one needs only use two sets of variables (the first two rows of $\bm{X}$). However, we will see below that the information contained in the ``extra'' part of $\Super_\rho$ plays an important role in understanding the global picture.

 For example, consider the hook $(a\,|\,b)=(2\,|\,0)$, so that $\Super_{(2\,|\,0)}=\M_{(2\,|\,0)}$. Then,
 \begin{align*}
\bm{D}_{(2\,|\,0)}(\bm{x})&=\det \begin{pmatrix} 
     \theta_1 & 1 &  x_1\\
    \theta_2 & 1&  x_2\\
    \theta_3 & 1  &  x_3
    \end{pmatrix}\\
   &=\theta_1(x_3-x_2)-\theta_2(x_3-x_1)+\theta_3(x_2-x_1).
 \end{align*}
 One readily checks that the module $\Super_{(2\,|\,0)}$ is spanned by the set of polynomials $\theta_1(u_3-u_2)-\theta_2(u_3-u_1)+\theta_3(u_2-u_1)$, one for each row $(u_1,u_2,u_3)$ of $\bm{X}$, together with the polynomials $\theta_1-\theta_2$ and $\theta_1-\theta_3$. Thus, its $\N^\infty$-graded Frobenius characteristic is equal to $\Super_{(2\,|\,0)} =s_{21}+ (q_1+q_2+\ldots)\,s_{3}$; and
\begin{equation}
   \Super_{(2\,|\,0)}=\nabla(\widehat{s}_3)=1\otimes s_{21}+ s_1\otimes s_{3}.
 \end{equation}
 Thus there is in this case actual equality between  $\Super_{(2\,|\,0)}$ and $\nabla(\widehat{s}_3)$.
The following table gives explicit calculated values which completes the picture\footnote{More values may be found in the appendix.} for all cases with $n\leq 4$. It confirms that Conjecture~\ref{Conj_Modules} holds in these instances; and, we see that the smallest case for which $\Super_{\rho}$ is ``larger'' than  $\nabla(\widehat{s}_\rho)$ is for $\rho=1111$.
\begin{itemize}\setlength\itemsep{4pt}
\item[]$\Super_{1}=\nabla(\widehat{s}_1)=1\otimes s_{1};$
	\smallskip
\item[]$\Super_{2}=\nabla(\widehat{s}_2)=1\otimes s_{11},$
\item[]$\Super_{11}=\nabla(\widehat{s}_{11})=1\otimes s_{2} + s_{1} \otimes s_{11};$
	\smallskip
\item[]$\Super_{21}= \nabla(\widehat{s}_{21})=s_{1} \otimes s_{21} + s_{2} \otimes s_{111},$
\item[]$\Super_{111}=\nabla(\widehat{s}_{111})=1\otimes s_{3} + (s_{1}  + s_{2}) \otimes s_{21} + (s_{11}+s_3) \otimes s_{111};$
	\smallskip
\item[]$\Super_{4}=\nabla(\widehat{s}_{4})=1\otimes s_{31} + s_{1} \otimes s_{22}  + (s_{1} + s_{2}) \otimes s_{211} + (s_{11}  + s_{3}) \otimes s_{1111},$
\item[]$\Super_{31}=\nabla(\widehat{s}_{31})=s_{1} \otimes s_{31}+ s_{2} \otimes s_{22}   + (s_{11}  + s_{2} + s_{3}) \otimes s_{211}  + (s_{21} + s_{4}) \otimes s_{1111},$
\item[]$\Super_{211}=\nabla(\widehat{s}_{211})=s_{2} \otimes s_{31} +(s_{11}+ s_{3}) \otimes s_{22}  +  (s_{21}  + s_{3}+ s_{4} )\otimes s_{211}  + (s_{31}  + s_{5}) \otimes s_{1111}\,$
\item[]$\begin{aligned} &\Super_{1111}=1\otimes s_{4} + (s_{1} + s_{2}+ s_{3}) \otimes s_{31} + (s_{2} + s_{21} + s_{4}) \otimes s_{22} \\
    &\qquad + (s_{11} + s_{21} + s_{31}  + s_{3}    + s_{4}   + s_{5}) \otimes s_{211} + (s_{111}+ s_{31}+ s_{41} +s_{6}  ) \otimes s_{1111}\\
    &\qquad =\nabla(\widehat{s}_{1111})+\rouge{s_{111}\otimes s_{1111}}.
\end{aligned}$
\end{itemize}
We observe, for values in this table, that we have
\begin{conjecture}[Skew]\label{Skew_Conj} For all $n$,
\begin{align}
   &(\Id \otimes\, e_1^\perp)\,\Super_{(n)} =\sum_{a=0}^{n-2} \Super_{(a\,|\, n-a-2)},\label{Skew2_Conj}\\
   &(e_1^\perp \otimes \Id)\,\Super_{1^n} = \sum_{a=1}^{n-1} \Super_{(a\,|\,n-a-1)}.\label{skewing_En}
\end{align}
\end{conjecture}
\noindent
These identities have been checked to hold for all $n\leq 6$.  
In particular, using Equations (1.7) and (1.10) of~\cite{HRS} 
and assuming a conjecture of~\cite{MRC1}  recalled further below as  Identity~\pref{delta_via_skewing_En}, we get that
\begin{theorem}
 The length $2$  restriction of~\pref{Skew2_Conj} holds, and Conjecture~\ref{Conj_Modules} implies that the length $2$  restriction of~\pref{skewing_En} is also true.
\end{theorem}
\begin{proof}[\bf Proof.]
To show the first equality, we observe that $L_{\leq 2}((\Id \otimes\, e_1^\perp)\,\Super_{(n)} )=(\Id \otimes\, e_1^\perp)\,L_{\leq 2}(\Super_{(n)} )$, since length restriction on the left-hand side of a tensor $s_\lambda\otimes_mu$ is clearly independent from operators acting on the right-hand side.
Moreover, from the point of view of representation theory, the operator $(\Id \otimes\, e_1^\perp)$ corresponds to restriction of the $\S_n$-action to the subgroup $\S_{n-1}$, of permutations that fix $n$. As discussed in~\cite{lattice}, $(\Id \otimes\, e_1^\perp)\,L_{\leq 2}(\Super_{(n)} )$ is simply the derivation-polarization span of the determinant
\begin{equation}\label{def_Determinant_restricted}
\bm{D}^-(\bm{x}):=\det \begin{pmatrix} 
     \bleu{\theta_1} & x_1 &\cdots  &  x_1^{n-2}\\
     \bleu{\theta_i}&  x_2 & \cdots &  x_2^{n-2}\\
     \vdots &\vdots  &  \ddots & \vdots \\
       \bleu{\theta_{n-1}}&  x_{n-1}  &\cdots &  x_{n-1}^{n-2}   
    \end{pmatrix},
 \end{equation}
Thus, to see that
   \begin{displaymath}
      (\Id \otimes\, e_1^\perp)\,L_{\leq 2}(\Super_{(n)} )=\sum_{a=0}^{n-2} L_{\leq 2}(\Super_{(a\,|\, n-a-2)}),
 \end{displaymath}
Next, assuming~\pref{conj_module_formule}, Identity~\pref{skewing_En} corresponds to
\begin{equation}\label{Skew_Conj_2}
  e_1^\perp \nabla(\widehat{s}_{n}) = \sum_{a=0}^{n-2} \nabla(\widehat{s}_{(a\,|\, n-a-2)}),
\end{equation}
This is shown to hold as follows. Indeed, using Formulas (I.12) from~\cite[page 368]{identity}, we get the operator identity $e_1^\perp \nabla =\nabla\, (e_1^\perp\Delta_{e_1} -\Delta_{e_1} e_1^\perp)$, with some rewriting. Hence equation~\pref{Skew_Conj_2} is equivalent to
   $$(e_1^\perp\Delta_{e_1} -\Delta_{e_1} e_1^\perp)\, \widehat{s}_{n} = \sum_{a=0}^{n-2} \widehat{s}_{(a\,|\, n-a-2)}.$$
 But this follows easily from
    $$\Delta_{e_1} \widehat{s}_{n} =  \sum_{a=0}^{n-1} \widehat{s}_{(a\,|\, n-a-1)},$$
  which, up to a multiplicative factor,  is Prop.~6.5 of~\cite{compositionalshuffle}.
 
Now, the length-$2$ restriction of the second identity corresponds, modulo Conjecture~\ref{Conj_Modules}, to the equality:
   $$L_{\leq 2} ((e_1^\perp \otimes \Id)\,\Super_{1^n})=\sum_{a=1}^{n-1} \nabla(\widehat{s}_{(a\,|\,b)}),$$
but~\pref{delta_via_skewing_En}, states that $\Delta'_{e_{n-2}} e_n=L_{\leq 2} ((e_1^\perp \otimes \Id)\,\Super_{1^n})$. Hence the second statement also holds.
 \end{proof}
   

 \subsection*{Links between  the \texorpdfstring{$\M_\rho$}{M}'s}
 We first recall that the Garsia-Haiman module $\mathcal{G}_\mu$ gives a representation theoretical interpretation for the combinatorial Macdonald polynomials $\widetilde{H}_\mu$. For any \define{diagram} (a finite subset $\bm{d}$ of $\N\times \N$), one may consider the determinant 
   $$\bm{D}_{\bm{d}}(\bm{x},\bm{y}):=\det(x_i^k y_i^\ell),$$
  with $1\leq i\leq n=\card\bm{d}$, and $(k,\ell)\in\bm{d}$. To make the sign of $\bm{D}_{\bm{d}}$ non-ambiguous, cells of ${\bm{d}}$ are ordered
so that $(k',\ell')\prec (k,\ell)$, if $\ell'>\ell$, or if $\ell'=\ell$ and $k'<k$. For any ${\bm{d}}$ we then consider the derivation closure  
 $$\mathcal{G}_{\bm{d}}=\Rational(\partial x_1,\ldots,\partial x_n;\partial y_1,\ldots,\partial y_n)\bm{D}_{\bm{d}}(\bm{x},\bm{y}).$$
When $\bm{d}$ is the Ferrers' diagram (with ``French'' convention) of a partition $\mu$. We recall that the following holds. 
\begin{theorem}[$n!$-Theorem, Haiman~\cite{haimanvanishing}] \label{factorial_theorem}
	The combinatorial Macdonald polynomial $\widetilde{H}_\mu(q,t;\bm{z})$ 
	is the bigraded Frobenius characteristic of the module $\mathcal{G}_{\mu}$.
\end{theorem}
Beside this case of Ferrers' diagrams, modules associated to Ferrers' diagrams ``missing'' one cell are (conjecturally) described in~\cite{lattice}. 
Observe that the determinant in~\pref{def_Determinant_hook} is obtained by replacing in $\bm{D}_{\bm{d}}$ the variables $\bm{y}$ by the inert variables $\bm{\theta}$,  for the diagram\footnote{Careful, this is not the Ferrer's diagram of the hook $(a\,|\,b)$.} 
   $$\bm{d}=\bm{d}{(a,b)}:=\{(i,0)\ |\ 0\leq i\leq a+b,\ {\rm and}\ i\not=a\} \cup \{(0,1)\},$$
which is illustrated in Figure~\ref{Fig_diagramme}. Thus, the module $\Mnk{(a\,|\, b)}{1}$ corresponds to the top $\bm{y}$-degree of $\mathcal{G}_{\bm{d}{(a,b)}}$; and this top degree is equal to $1$.
Its (simply)-graded\footnote{Obtained by setting $q_1=q$, and all other $q_i$ equal to $0$.} Frobenius characteristics  is thus
  $$\M_{(a\,|\, b)}(q;\bm{z})=\mathcal{G}_{\bm{d}{(a,b)}}(q,t;\bm{z})\big|_{{\rm coeff}\, t}.$$
\begin{figure}[ht]
\begin{center}
 \begin{tikzpicture}[thick,scale=.4]
 \draw[->] (-.5,-.5) -- (-.5,3);
  \draw[->] (-.5,-.5) -- (10,.-.5);
  \node at (-1.2,2.1) {$\scriptstyle \N$};
    \node at (10.2,-1) {$\scriptstyle \N$};
\node at (3,-1) {$\scriptstyle a$};
\carrevert{0}{1}
\carrejaune{0}{0}
\carrejaune{1}{0}\carrejaune{2}{0} 
    \carrejaune{4}{0}\carrejaune{5}{0}\carrejaune{6}{0}\carrejaune{7}{0}\carrejaune{8}{0}
\end{tikzpicture}
\end{center}
\qquad  \vskip-25pt
\caption{The diagram $\bm{d}{(a,b)}$.}
\label{Fig_diagramme} 
\end{figure}
In particular, the diagram $\bm{d}(n-1,0)$ happens to be the hook-shape $(n-2\,|\, 1)$. In view of $n!$-Theorem above we get
 \begin{align}
      \M_{(n-1\,|\, 0)}(q;\bm{z}) &= \widetilde{H}_{(n-2\,|\, 1)} (q,t;\bm{z})\big|_{{\rm coeff}\, t}\nonumber\\
         &=\bm{H}_{(n-2\,|\, 1)} (q;\bm{z}),\label{cas_equerre}
 \end{align}
with the right-hand side of the above identity being an instance, with $\mu={(n-2\,|\, 1)}$, of symmetric polynomials that we denote by $\bm{H}_\mu(q;\bm{z})$. These are directly related to the dual Hall-Littlewood symmetric functions $Q_\mu'$ (following Macdonald's~\cite[Exer.7, page 234]{macdonald} notation; see also~\cite{HRS}), which are such that
    $$Q_\mu'(q;\bm{z}) =\sum_{\lambda} K_{\mu\lambda}(q)\, s_\lambda(\bm{z}),$$
 where the $K_{\lambda\mu}(q)\in \N[q]$ are the Kostka-Foulkes polynomials. More precisely, we have $\bm{H}_\mu =\omega Q'_{\mu'}$, so that
\begin{equation}
    \bm{H}_\mu=\sum_{\lambda} K_{\mu'\lambda'}(q)\, s_\lambda(\bm{z}).
 \end{equation}
Now, letting $\pi$ stand for the operator $ \pi:=\sum_{i=1}^{n}{\partial x_i}$,
it is clear that $\pi \bm{D}_{(a\,|\, b)}(\bm{x})=\bm{D}_{(a+1\,|\, b-1)}(\bm{x})$. It follows that we have the following projection, and associated degree 1 reducing isomorphism:
$$\xymatrix{
  \Mnk{(a\,|\, b)}{1}   \ar@{->>}[r]^-{\pi} &  \Mnk{(a+1\,|\, b-1)}{1}},\qquad {\rm and}\qquad
 \xymatrix{\Super_{(a\,|\, b)}^{\langle 1\rangle}   \ar@{->}[r]^-{\sim} &  \Super_{(a+1\,|\, b-1)}^{\langle 1\rangle}}.
   $$
This last isomorphism translates into the equality $\Super_{(a\,|\, b)}(q;\bm{z}) = q\, \Super_{(a+1\,|\, b-1)}(q;\bm{z})$, and we conclude by~\pref{cas_equerre} that
\begin{equation}
    \M_{(a\,|\,b)}(q;\bm{z})=(1+q+\ldots + q^{b})\,\bm{H}_{(n-2\,|\,1)}(q;\bm{z}).
 \end{equation} 
It is interesting to notice that, together with Conjecture~\ref{Conj_Modules}, a particular case of Identity~1.7 of~\cite{HRS} also gives
 \begin{equation}
  \M_{(1\,|\, n-2)}(q;\bm{z})  = \Delta'_{e_{n-2}} (q,0;\bm{z}).
 \end{equation}
 Moreover, using~\pref{combinatorial_nabla}, we have the LLT-polynomial expression
 \begin{align}
    \nabla(\widehat{s}_{(a\,|\, b)})(q,0;\bm{z}) &=  \mathbb{L}_{\Gamma_a}(q;\bm{z})
 \end{align}
   since $\Gamma_a$ is the only Dyck path for which $\area(\gamma)-{a}=0$.
From all this, we get the following.
  \begin{prop} Conjecture~\ref{Conj_Modules} holds when $t=0$.
 \end{prop}
 
 
\subsection*{Link to intersection of Garsia-Haiman modules} The main objective of paper~\cite{ScienceFiction} is to describe the decomposition of families of Garsia-Haiman modules indexed by partitions of $n$ (covered by a given partition $\mu$ of $n+1$), with respect to their relative intersections. In the particular case when $\mu=(n-1\,|\,1)$, one may thus consider the two hook partitions $(n-1,1)=(n-2\,|\, 1)$ and $(n)=(n-1\,|\,0)$. A special case of the conjectures stated therein, asserts that the bi-graded Frobenius of the intersection $\mathcal{I}_n:=\mathcal{G}_{\bm{d}{(n-2,1)}}\cap  \mathcal{G}_{\bm{d}{(n-1,0)}}$ is given by the formula 
\begin{equation}
    \mathcal{I}_n(q,t;\bm{z}) = \frac{q^{n-1} \widetilde{H}_{(n-1,1)}-t \widetilde{H}_{(n)}}{q^{n-1}-t}.
\end{equation}
Moreover, still assuming conjectures of~\cite{ScienceFiction}, the module $\mathcal{G}_{\bm{d}{(n-2,1)}}$ decomposes as
   $\mathcal{G}_{\bm{d}{(n-2,1)}}=\mathcal{I}_n\oplus \mathcal{I}_n^\perp$,
 with
   $$\mathcal{I}_n^\perp = \{f(\partial \bm{x}) \bm{D}_{\bm{d}(n-2,1)}(\bm{x})\ |\ f(\partial \bm{x})\in \mathcal{I}_n\}.$$
It may be shown that both $\mathcal{I}_n$ and $\mathcal{I}_n^\perp$ are of dimension $n!/2$.
It follows that 
\begin{align*}
    \mathcal{I}_n^\perp(q,t;\bm{z})&= \widetilde{H}_{(n-1,1)}(q,t;\bm{z}) - \frac{q^{n-1} \widetilde{H}_{(n-1,1)}(q,t;\bm{z})-t \widetilde{H}_{(n)}(q,t;\bm{z})}{q^{n-1}-t}\\
       &=\frac{t}{q^{n-1}-t}(\widetilde{H}_{(n-1,1)}(q,t;\bm{z})-\widetilde{H}_{(n)}(q,t;\bm{z}))\\
       &=t\,\bm{H}_{(n-1,1)}(q,\bm{z}).
\end{align*}
Moreover, $\mathcal{I}_n(q,t;\bm{z})= q^{\binom{n-1}{2}}\omega \bm{H}_{(n-1,1)}(1/q,\bm{z})$, so that we get
\begin{equation}\label{decomp_ScienFiction}
  \widetilde{H}_{(n-1,1)}(q,t;\bm{z}) =q^{\binom{n-1}{2}}\omega \bm{H}_{(n-1,1)}(1/q,\bm{z})+t\,\bm{H}_{(n-1,1)}(q,\bm{z}).
\end{equation}
For example, with $n=5$, 
we get
  \begin{align*}
  \widetilde{H}_{41} & = (q^{6}s_{2111} + (q^{4} + q^{5})s_{221} + (q^{3} + q^{4} + q^{5})s_{311} + (q^{2} + q^{3} + q^{4})s_{32} + (q + q^{2} + q^{3})s_{41} + s_{5})\\
 &+t\,(s_{41}+ (q^{2} + q)s_{32} + (q^{3} + q^{2} + q)s_{311} + (q^{4} + q^{3} + q^{2})s_{221}+ (q^{5} + q^{4} + q^{3})s_{2111}
 + q^{6}s_{11111}   ),
 \end{align*} 
where each term $q^a s_\nu$ in the top row, corresponds to a term $tq^{6-a} s_{\nu'}$ in the bottom row. Setting $q=1$, we may check that \pref{decomp_ScienFiction} equality specializes to
 $$\widetilde{H}_{(n-1,1)}(1,t;\bm{z}) =h_{21^{n-2}}(\bm{z})+t\,e_{21^{n-2}}(\bm{z}).$$
  We may use these observations to explicitly construct a basis of $\Snk{(a\,|\,b)}{1}$ in the following manner. Let $\mathcal{B}_n$ be a basis of $\mathcal{I}_n$, then
  \begin{lemma}
  For each $a\geq 1$, with $b=n-a-1$ and $\bm{d}:={\bm{d}(a,b)}$,  the set $\{\varphi(\partial \bm{x}) \bm{D}_{\bm{d}}(\bm{x})\ |\ \varphi( \bm{x})\in \mathcal{B}_n\}$ forms a basis of $\Snk{(a\,|\,b)}{1}$.
  \end{lemma}
 \begin{proof}
 We only check that the proposed set is linearly independent. By hypothesis, we already know  that any $\varphi(\bm{x})$ in the span of $\mathcal{B}_n$ is of the form $\varphi(\bm{x})=\psi(\partial{\bm{x}}) \bm{D}_{\bm{d}}(\bm{x})$, for some polynomial $\psi(\bm{x})$. We may verify that $\varphi(\partial \bm{x}) \bm{D}_{\bm{d}}(\bm{x})\not=0$, as follows. 
 By commutation, 
     $$\psi(\partial \bm{x}) \left(\varphi(\partial \bm{x}) \bm{D}_{\bm{d}}(\bm{x}) \right)= 
      \varphi(\partial \bm{x}) \left(\psi(\partial \bm{x}) \bm{D}_{\bm{d}}(\bm{x}) \right)= \varphi(\partial \bm{x})  \varphi( \bm{x}). $$
This last expression does not vanish, since its constant term is the sum of the square of coefficients of $ \varphi( \bm{x})$, hence is not equal to $0$. It follows that the map sending $\varphi(\bm{x})$ to $\varphi(\partial \bm{x}) \bm{D}_{\bm{d}}(\bm{x})$ is injective, and we get the expected linear independence.
 \end{proof}
Recall that $\Mnk{(n)}{1}$, which is the classical module of $\S_n$-harmonic polynomials, is known to decompose into irreducibles given by (homogenous) higher Specht modules (see~\cite{higher}). One may readily exploit this to construct a basis $\mathcal{B}_n$ which reflects the fact that $\mathcal{I}_n$ is a graded sub-module of  $\Mnk{(n)}{1}$.


 \subsection*{Hook components conjecture} Our second conjecture describes a link between the alternant isotypic component of $\Super_{\rho}$ and the isotypic components corresponding to hooks $(a\,|\,b)$.  We have:
 \begin{conjecture}[Hook Components] For all $\rho$, and all $0\leq a\leq n-1$, the hook-component coefficients are obtained as
 \begin{equation}\label{conj_hook}
    \coeff_{\rho,(a\,|\,b)} = e_{a}^\perp \A_\rho.
 \end{equation}
 \end{conjecture}
 Now, observe that the equality $\A_{(n+1)}= \A_{1^{n}}$ follows readily from the definition of the module $\Super_\rho$.
 Thus we may deduce, using~\pref{conj_hook}, that for all $a$ 
   \begin{equation}
    \coeff_{(n\,|\,0),(a\,|\,b+1)}=  \coeff_{{(0\,|\,n-1)},(a\,|\,b)}.
    \end{equation}
 

\subsection*{Length conjecture}
One of the interesting implications of this theorem, together with Conjecture~\ref{length_conjecture} below, is that we can reconstruct $\A_\rho$ from (very) partial knowledge of the values of 
the $\langle \Super_\rho, s_\mu\rangle$. To see how this goes, let us first state the following conjecture, defining the \define{length} $\ell(f)$ of a symmetric function $f$, to be the maximum number of parts $\ell(\lambda)$ in a partition $\lambda$ that index a Schur function $s_\lambda$ occurring with non-zero coefficients $a_\lambda$ in its Schur expansion $f=\sum_\lambda a_\lambda f_\lambda$. In formula: 
    $$\ell(f) =\max_{a_\lambda\not=0} \ell(\lambda).$$
The following conjecture extends to all hooks $\rho$, a similar conjecture (see~\cite[Conj.~3]{MRC1}) for the $\Super_{1^n}=\E_n$.
\begin{conjecture}[Coefficients-Length]\label{length_conjecture} If $\rho=(a\,|\,b)$ with $a\geq 1$, then we have 
   \begin{equation}
       \ell(\coeff_{\rho\mu}) \leq  n-\mu_1,
     \end{equation}
  for all partition $\mu$ of $n$. 
\end{conjecture}

In particular, when $\mu=(n-2\,|\, 1)$, the length of $\langle \Super_\rho, s_{\mu}\rangle$ is conjectured to be bounded by $1$. As it happens, we have
\begin{align}
     & \langle \Super_{(a\,|\, b)}, s_{(n-1\,|\,1)}\rangle= \langle \nabla(\widehat{s}_{(a\,|\, b)}), s_{(n-1\,|\,1)}\rangle =0,\qquad {\rm and}\label{coefficient_of_n}\\
    & \langle \Super_{(a\,|\, b)}, s_{(n-2\,|\,1)}\rangle= \langle \nabla(\widehat{s}_{(a\,|\, b)}), s_{(n-2\,|\,1)}\rangle = s_{b}\label{coefficient_of_n1}.
 \end{align}
Indeed (see~\cite{MRC1}), we already know that
\begin{align}  
   &\langle \Super_{1^n}, s_{(n-1\,|\,0)}\rangle= \langle \nabla(e_n), s_{(n-1\,|\,0)}\rangle = 1 ,\qquad {\rm and}\\
   &\langle \Super_{1^n}, s_{(n-2\,|\,1)}\rangle= \langle \nabla(e_n), s_{(n-2\,|\,1)}\rangle = s_1+s_2+\ldots+ s_{n-1}.
 \end{align}
Thus we obtain formulas~\pref{coefficient_of_n} and~\pref{coefficient_of_n1} (together with the above), by respectively taking coefficients of $s_{(n-1\,|\,0)}$  and $s_{(n-2\,|\,1)}$ on both sides of ~\pref{skewing_En}.


\subsection*{Reconstruction of Hilbert series of alternants}
Let us illustrate, assuming~\pref{conj_hook} together with the Coefficient-Length Conjecture, how we may reconstruct\footnote{A similar reconstruction, for the case when $\rho=1^n$, is described in~\cite{MRC1}. } $\A_\rho$, when $\rho=(a\,|\,b)$ for $a\geq 1$. First, we have
  \begin{equation}
 e^\perp_{n-1} \A_\rho = 0,
    \end{equation}
 so that $\A_\rho$ contains no terms\footnote{Recall that the Schur expansion of $\A_\rho$ only has positive integer coefficients.} of length larger or equal to $n-1$.
Next, using~\pref{coefficient_of_n1}, we get that
   \begin{equation}
 e^\perp_{n-2}\, \A_\rho = \coeff_{(a\,|\,b),(n-2\,|\, 1)}=s_{b},
    \end{equation}
from which we infer that
   \begin{equation}\label{first_terms_of_Arho}
   \A_\rho=\rouge{s_{(b\,|\,{n-3})}}+\underbrace{\quad\ldots\quad }_{\hbox{\tiny lower length terms}} ,
   \end{equation}
 Likewise, all terms of length $n-3$ of $\A_\rho$ are imposed by the identity
\begin{equation}\label{troisieme_approximation_An}
   e_{n-3}^\perp\, \A_\rho = \langle \Super_\rho,s_{(n-3\,|\, 2)}\rangle = \langle \nabla(\widehat{s}_\rho),s_{(n-3\,|\, 2)}\rangle,
 \end{equation}
since $ \langle \Super_\rho,s_{(n-3\,|\, 2)}\rangle$ is of at most length $2$, hence its value is entirely characterized by that of $\nabla(\widehat{s}_\rho)$.
For instance, for hooks of size $6$, we may calculate explicitly that 
   \begin{align*}
       &\langle \nabla(\widehat{s}_{(6)}),s_{411}\rangle= \rouge{s_{1}} + s_{2} + s_{3} + s_{4},\\
       &\langle \nabla(\widehat{s}_{51}),s_{411}\rangle =\rouge{s_{11}}  + {s_{21}} + s_{31}  + \rouge{s_{2}}+ s_{3} + s_{4} + s_{5},\\
       &\langle \nabla(\widehat{s}_{411}),s_{411}\rangle =\rouge{s_{21}} + {s_{31}} + s_{41}  +  s_{22} + \rouge{s_{3}} + s_{4} + s_{5} + s_{6},\\
       &\langle \nabla(\widehat{s}_{3111}),s_{411}\rangle =\rouge{s_{31} } + {s_{41}} + s_{51} +s_{32} +  \rouge{s_{4}} + s_{5} + s_{6} + s_{7},\\
       &\langle \nabla(\widehat{s}_{21111}),s_{411}\rangle = \rouge{s_{41}} + {s_{51}} + s_{61} + s_{32} + s_{42} + \rouge{s_{5}} + s_{6} + s_{7} + s_{8};
     \end{align*}
from which we deduce all terms of $\A_\rho$ of length larger or equal to $3$. This gives 
   \begin{align*}
       &\A_{(6)} = \rouge{s_{1111} }+ s_{311} + s_{411} + s_{511} + \langle\nabla(\widehat{s}_{(6)}),e_6\rangle,\\
      &\A_{51}  = \rouge{s_{2111}} + s_{321} + s_{421} + s_{411} + s_{511} + s_{611}+ \langle\nabla(\widehat{s}_{51}),e_6\rangle ,\\
      &\A_{411} = \rouge{s_{3111}} + s_{331} + s_{421} + s_{521} + s_{511} + s_{611} + s_{711} +\langle \nabla(\widehat{s}_{411}),e_6\rangle ,\\
      &\A_{3111} = \rouge{s_{4111}} + s_{431} + s_{521} + s_{621} + s_{611} + s_{711} + s_{811} + \langle\nabla(\widehat{s}_{3111}) ,e_6\rangle,\\
      &\A_{21111} = \rouge{s_{5111}} +s_{431} +  s_{531} + s_{621} + s_{721}  + s_{711} + s_{811} + s_{911}+\langle \nabla(\widehat{s}_{21111}),e_6\rangle.
     \end{align*}
  in which the first terms correspond to~\pref{first_terms_of_Arho}. Observe that some of the terms in $ \langle \nabla(\widehat{s}_{\rho}),s_{411}\rangle$ are already obtained by skewing by $e_3$ the length-$4$ terms in the $\A_{\rho}$'s. Hence, we only need to add the necessary length-$3$ terms to account for the ``missing'' terms. We can then conclude the entire construct by adding $\langle\nabla(\widehat{s}_\rho),e_6\rangle$, since it contains precisely the terms of length less or equal to $2$ that should appear in $\A_\rho$.


 \section{\bleu{Structure properties of \texorpdfstring{$\Super_{\rho}$}{f}}} 
We will currently see how the explicit description of $\Super_{\rho}$ is easier to obtain in view of the identities either proven or conjectured above.
This is similar to properties $\E_n=\Super_{1^n}$, for which we have conjectured the following.  For $0\leq j\leq n-1$, we conjectured in~\cite{MRC1} that
\begin{equation}\label{delta_via_skewing_En}
    (\eperp{j}\,\E_n)(q,t;\bm{z})= \Delta'_{e_{n-1-j}}\,e_n(\bm{z}).
    \end{equation}
 

\subsection*{Reduced length components}\label{section_length}
The \define{length-$d$ component} of $\Super_\rho$ is set to be
\begin{equation}
     \Super_\rho^{(d)}:=\sum_{\mu\vdash n} \coeff_{\rho\mu}^{(d)}\otimes s_\mu,\qquad {\rm with}\qquad 
     	\coeff_{\rho\mu}^{(d)}=\sum_{\ell(\lambda)=d} c_{\lambda\mu}^\rho s_\lambda.
 \end{equation}
We clearly have
  \begin{equation}
     \Super_\rho:=\Super_\rho^{(0)}+\Super_\rho^{(1)}+\ldots \Super_\rho^{(l)},
 \end{equation}
 where $l=\ell(\Super_\rho)$ is the maximal length occurring in terms of $\Super_\rho$. With this notation, Conjecture~\ref{length_conjecture} states that $\ell(\coeff_{\rho\mu})=|\mu|-\mu_1$.
For our discussion, it will be handy to consider the \define{reduced length-$d$ components} $\sigma_\rho^{(d)}$ of $\Super_\rho$, defined as
\begin{equation}\label{reduced_length}
   \sigma_\rho^{(d)}:=\eperp{d}\,\Super_\rho^{(d)}.
 \end{equation}
 We also set
  \begin{equation}\label{reduced_length_coeffs}
      \gamma_{\rho\mu}^{(d)}:=e_d^\perp\coeff_{\rho\mu}^{(d)}.  \qquad {\rm and}\qquad \alpha_\rho^{(d)}:=e_d^\perp\A_\rho^{(d)}.
  \end{equation}
Clearly, the knowledge of $\sigma_n^{(d)}$ is equivalent to that of $\Super_\rho^{(d)}$.
For example, the non-vanishing reduced length components of $\Super_{41}$ are
 \begin{align*}
\sigma^{(1)}_{41}&=1\otimes s_{41}  + (s_{1}+ s_{2}) \otimes s_{32} + (s_{1} + s_{2} + s_{3}) \otimes s_{311}+ (s_{2} + s_{3} + s_{4}) \otimes s_{221} \\
    &\qquad 
     + (s_{3} + s_{4} + s_{5}) \otimes s_{2111} + s_{6} \otimes s_{11111},\\
\sigma^{(2)}_{41}&= 1 \otimes s_{32} + (1+s_{1}) \otimes s_{311}  + (2\,s_{1}+s_{2}) \otimes s_{221} + (s_{11}+s_{1}+2\,s_{2} + s_{3}) \otimes s_{2111},\\
    &\qquad 
     +(s_{21}+s_3+s_4) s_{11111}\\
\sigma^{(3)}_{41}&=1 \otimes s_{2111}+s_1 \otimes s_{11111}.
\end{align*}
It may be shown that
\begin{prop}
For all $\rho=(a\,|\,b)$, with $a\geq 1$, we have (almost?)
   \begin{equation}
        \sigma^{(n-2)}_\rho=\sum_{i=0}^{\min(b,\lfloor n/2\rfloor)} s_{b-i}\otimes s_{2^i 1^{n-2i}}.
   \end{equation} 
\end{prop}


\subsection*{The \texorpdfstring{$e$}{e}-positivity phenomenon} As discussed in~\cite{open}, most of the symmetric functions constructed via the elliptic Hall algebra approach  exhibit a $e$-positivity when specialized at $t=1$. 
 We consider here the case of $\Super_\rho$, for which we get the specialization of any one of the (infinitely many) parameters $q_i$ to the value $1$. This is obtained via the plethystic evaluation at $1+\bm{q}$ of the $\GL_\infty$-coefficients $\coeff_{\rho\mu}$ of $\Super_{\rho}$. Noteworthy is the fact that this operation is invertible, as long as there are infinitely many parameters. 
 
 It is worth underlying the difference between 
\begin{align*}
     &p_j[1+q_1+q_2+\ldots +q_k+\ldots ] =1+q_1^j+q_2^j+\ldots +q_k^j+\ldots,\qquad {\rm and}\\ 
     &p_j[q_1+q_2+\ldots +q_k+\ldots ]\Big|_{q_1\to 1+q_1} = (1+q_1)^j+q_2^j+\ldots +q_k^j+\ldots 
\end{align*}
We see here the difference between the two possible orders in which we may apply the operators $p_j[-]$, and substitution of $1+q_1$ for $q_1$. The $e$-positivity phenomenon considered below is for the first of these, in contrast with similar results  that appeared in~\cite{Alexandersson,Panova,dadderio,Romero}, in which the second order of application of the operators is considered.
  
For the sake of discussion, let us write 
\begin{equation}
   \F_{\rho}:= \Super_\rho[1+\bm{q};\bm{z}],
\end{equation}
and write
\begin{align*}
    \F_{\rho}&=\sum_{\mu\vdash n}  \coeff_{\rho\mu}[1+\bm{q}]\otimes s_\mu(\bm{z}); 
\end{align*}
or equivalently in Schur$\otimes$Schur-format:
\begin{align}
    \F_{\rho}&=\sum_{\nu\vdash n} \coeffprime_{\rho\nu}\otimes e_\nu\label{e_expression},
\end{align}
with $\coeffprime_{\rho\nu}$  the coefficients of $e_\nu(\bm{z})$ in $\F_{\rho}$. Then, as far as we can check experimentally, all of the $\coeffprime_{\rho\nu}$ are  \define{Schur-positive}.
   For instance, we have
 \begin{align*}
     \F_{41}&=(s_{211} + s_{32} + s_{41} + s_{51} + s_{7})\otimes e_{5}\\
&\qquad +(s_{111} + s_{22} + s_{11} +  s_{21} + s_{3} + 2s_{31} + s_{4} + s_{41} + s_{5} + s_{6})\otimes e_{41}\\
&\qquad +(2s_{21} + s_{31} + s_{3} + s_{4} + s_{5})\otimes e_{32}  +(s_{11} + s_{21} + s_{1} + 2s_{2} + s_{3} + s_{4})\otimes e_{311}\\
&\qquad+ (s_{11} + s_{1} + 2s_{2} + s_{3})\otimes e_{221} + (1+ s_{1})\otimes e_{2111}\end{align*}
By definition, the $\coeff_{\rho\mu}$ are related to the $\coeffprime_{\rho\nu}$ by the identity
  \begin{equation}
      \coeff_{\rho\mu}[1+\bm{q}]=\sum_{\nu} K_{\mu'\lambda}\coeffprime_{\rho\lambda},
  \end{equation}
  where the $K_{\mu\lambda}$ are the usual Kostka numbers. 

There are close ties between this $e$-positivity phenomenon and our conjectures. To see this, recall that  the coefficient of $e_n$ in the $e$-expansion of $s_\mu$ vanishes for all $\mu$ except hooks; and it is known to be equal to $(-1)^k$ when $\mu=(k\,|\,j)$, with $n=k+j+1$. 
 Since the forgotten symmetric functions $f_\nu$ are dual to the $e_\nu$, we may write this as 
    $$\langle s_\mu,f_n\rangle=\begin{cases}
      (-1)^k, & \text{if}\ \mu=(k\,|\,j), \\
       0 & \text{otherwise}.
\end{cases}$$
We may then calculate, using~\pref{conj_hook}, that
\begin{align*}\coeffprime_{\rho,(n)}&= \langle \Super_\rho[1+\bm{q};\bm{z}],f_{n}\rangle
   =\sum_{\mu\vdash n}  \coeff_{\rho\mu}[1+\bm{q}]\,  \langle s_\mu, f_n\rangle\\
   &=\left(\sum_{k=0}^{n-1}  (-1)^k\, \coeff_{\rho,(k\,|\,j)}\right)[1+\bm{q}]
   =\left(\sum_{k\geq 0}  (-1)^k\, e_k^\perp \A_{n}\right)[1+\bm{q}] .  
\end{align*}
For any symmetric function $F$, one has $ \sum_{k\geq  0} (-1)^k\,e_k^{\perp}F(\bm{q})=F[\bm{q}-1]$. Thus,
we find that $\coeffprime_{\rho,(n-1\,|\,0)} = (\A_{\rho}[\bm{q}-1])[1+\bm{q}]=\A_{\rho}$, and we conclude the following.
\begin{prop} The Hook Conjecture implies that, for all $\rho$, the coefficient $ \coeffprime_{\rho,(n)}$ of   $e_n$ in $\F_\rho$ is Schur positive. 
\end{prop}
To get more, let $\mu$ be any partition of $n$ which is largest in dominance order among those such that $\coeff_{\rho\mu}\not=0$. Then, it is easy to see that 
\begin{equation}\label{cas_extremes}
    \coeffprime_{\rho\mu'}=\coeff_{\rho\mu}[1+\bm{q}].
  \end{equation}
We thus automatically have  Schur-positivity of $\coeffprime_{\rho\mu'}$. 
Experiments suggest that, for all hooks $\rho=(a\,|\,b)$ and  $\mu=(k\,|\,j)$, if $m:=\min(j,k)$ then we have
   \begin{equation}
      \coeffprime_{\rho\mu} = \sum_{i=0}^{m}   \coeff_{(a\,|\,b-i),(k-i\,|\, j)},
    \end{equation}
except when $\rho=\mu=1^n$, in which case we simply have $\coeffprime_{\rho\mu} =1$.


\subsection*{Trivariate shuffle conjecture} 
The trivariate shuffle conjecture of~\cite{trivariate}, corresponding below to $\rho=(0\,|\, n-1)$, may at least be extended to other cases as follows. Recall definition~\pref{def_Gamma_a} of $\Gamma_a$.
\begin{conjecture}[Trivariate shuffle] For hooks $\rho=(a\,|\,b)$, with $a$ equal to either $0$, $1$, or $n-1$, we have
\begin{equation}
   \Super_\rho(q,t,1;\bm{z}) =\sum_{\Gamma_a\leq \alpha\leq \beta} 
           q^{d(\alpha,\beta)} \mathbb{L}_\beta(t;\bm{z}),
  \end{equation}
  where the Dyck path $\alpha$ lies below the Dyck path $\beta$ in the Tamari poset, and $d(\alpha,\beta)$ is the length of the longest strict chain going from $\alpha$ to $\beta$ in this poset.
\end{conjecture}
 
Again, we underline that the case $a=0$ already appears in~\cite{trivariate}, and that the case $a=n-1$ is more or less implicit in~\cite{bousquet}. We expect that some variant of this formula should hold for other hooks, maybe with some tweak to the LLT-polynomial part. It would also be nice to have similar expressions involving $r$, for $\Super_\rho(q,t,r;\bm{z})$, but this is not known. 

%


\subsection*{More observed properties}
Recall that Identity~4.17 in Theorem 4.2 of~\cite[Identity~4,4]{identity} states (in our notations) that for $a+b=n$
    $$\scalar{\nabla(\widehat{s}_{n+1})}{s_{(a\,|\, b)}} = \scalar{\nabla(e_n)}{s_{(a\,|\, b-1)} }.$$
 This equality appears to lift to the following similar multivariate identity:
\begin{align}
  & \scalar{\Super_{(n+1)}}{s_{(a\,|\, b)}} =
   \scalar{\Super_{1^n}}{s_{(a\,|\, b-1)}},\qquad \hbox{for all}\qquad a+b=n.
    \end{align} 
Another interesting observed identity is:
\begin{align}
   & (e_1^\perp \A_{\rho})(\bm{q}) = \sum_{\mu\vdash n} \scalar{\Super_{\rho}(1+\bm{q};\bm{z})}{f_\mu(\bm{z})}.
\end{align}

\bigskip
\noindent
{\bf{Acknowledgments.}}\  Much of this work would not have been achieved without the possibility of perusing sufficiently large expressions resulting from difficult  explicit computations. These calculations were very elegantly realized by Pauline Hubert and Nicolas Thi\'ery using the Sage computer algebra system, together with an inspired used of the right mathematical properties of objects considered, and  special properties of higher Specht modules. Indeed, direct calculations of the relevant symmetric function expressions rapidly become unfeasible, even with powerful computers. 

\renewcommand{\appendixpagename}{\bleu{Appendix}}
\section*{\bleu{Appendix}}

{\footnotesize
\subsection*{The \texorpdfstring{$s_\lambda\otimes s_\mu$}{ss}-expansions of \texorpdfstring{$\nabla(\widehat{s}_{\rho})$}{nsmn} and \texorpdfstring{$\Super_{\rho}$}{emn} for hooks of size \texorpdfstring{$5$}{5}.}\

\parindent=0pt

\begin{itemize}\setlength\itemsep{6pt}\setlength{\itemindent}{-.5in}

\item[]$\begin{aligned}    
\nabla(\widehat{s}_{5})&=
  1 \otimes s_{41} 
   + (s_{1} + s_{2}) \otimes s_{32} + (s_{1} + s_{2} + s_{3}) \otimes s_{311} 
  + (s_{11} + s_{21}+ s_2+ s_{3} + s_{4}) \otimes s_{221}\\
         &\qquad   
  + (s_{11}  + s_{21} + s_{31}   + s_{3}  + s_{4} + s_{5}) \otimes s_{2111}
  + (s_{31}  + s_{41}   + s_{6})\otimes s_{11111}
  \end{aligned}$ 
  
\item[]$\begin{aligned}    
\nabla(\widehat{s}_{41})&=
    s_{1} \otimes s_{41} 
 + (s_{11} + s_{2} + s_{3})\otimes s_{32} 
 + (s_{11}  + s_{21} + s_{2}+ s_{3} + s_{4}) \otimes s_{311}\\
         &\qquad
 + (2s_{21}  + s_{31}+ s_{3} + s_{4} + s_{5}) \otimes s_{221} 
 + (s_{21}  + s_{22}  + 2s_{31}+ s_{41}    + s_{4} + s_{5} + s_{6}) \otimes s_{2111}\\
         &\qquad
 + (s_{32}  + s_{41}  + s_{51}  + s_{7})\otimes s_{11111}
 \end{aligned}$
 
 \item[]$\begin{aligned}  
\nabla(\widehat{s}_{311})&=
    s_{2} \otimes s_{41} + (s_{21}+ s_{3} + s_{4}) \otimes s_{32}
  + (s_{21}+ s_{31}+ s_{3}+ s_{4} + s_{5})\otimes s_{311}
  	\\ &\qquad 
  + (s_{22}+ 2s_{31} + s_{41} + s_{4} + s_{5}+ s_{6}) \otimes s_{221}
  + (s_{31}    + s_{32}   + 2s_{41}+ s_{51} + s_{5} + s_{6} + s_{7}) \otimes s_{2111}
  	\\   &\qquad 
  + (s_{42}   + s_{51}  + s_{61} + s_{8}) \otimes s_{11111}\\
 \end{aligned}$
 
  \item[]$\begin{aligned}  
\nabla(\widehat{s}_{2111})&=
   s_{3} \otimes s_{41} + (s_{21} + s_{31} + s_{4} + s_{5}) \otimes s_{32} 
+ (s_{22}+ s_{31} + s_{41} + s_{4} + s_{5} + s_{6}) \otimes s_{311}\\
         &\qquad  
+ (s_{32} + s_{31} + 2s_{41} + s_{51} + s_{5} + s_{6} + s_{7}) \otimes s_{221}
	\\  &\qquad 
+ (s_{32} + s_{42} + s_{41} + 2s_{51} + s_{61}  + s_{6} + s_{7} + s_{8}) \otimes s_{2111} 
+ (s_{33}  + s_{52} + s_{61}   + s_{71}  + s_{9}) \otimes s_{11111}
\end{aligned}$

\end{itemize}

\begin{itemize}\setlength\itemsep{6pt}\setlength{\itemindent}{-.5in}
\item[] $\Super_{5}=\nabla(\widehat{s}_5)+s_{111} \otimes s_{11111},$
\item[] $\Super_{41}=\nabla(\widehat{s}_{41})+s_{111} \otimes s_{2111} + s_{211} \otimes s_{11111},$
\item[] $\Super_{311}=\nabla(\widehat{s}_{311})+s_{111} \otimes s_{221} + s_{211} \otimes s_{2111} + s_{311} \otimes s_{11111},$
\item[] $\Super_{2111}=\nabla(\widehat{s}_{2111})+s_{211} \otimes s_{221} + s_{311} \otimes s_{2111} + s_{411} \otimes s_{11111},$
\item[]  $\begin{aligned}       
\Super_{11111}&= \nabla(\widehat{s}_{11111}) + ( s_{211} + s_{311}) \otimes s_{221}
	+ (s_{111} + s_{211} + s_{311} + s_{411}) \otimes s_{2111}
	+ (s_{1111} + s_{311} + s_{411} + s_{511}) \otimes s_{11111}.
\end{aligned}$
 \end{itemize}       
   
  (The value of $\nabla(\widehat{s}_{11111}))$ may be found in~\cite{MRC1}.)
  


\subsection*{The expansions of \texorpdfstring{$ \Super_{\rho}$}{emn} for hooks of size \texorpdfstring{$6$}.}
\ \null\medskip

\begin{itemize}\setlength\itemsep{6pt}\setlength{\itemindent}{-.5in}
\item[] $\begin{aligned}  
\Super_{6}&= 
1 \otimes s_{51}
+ ( s_{1} + s_{2} + s_{3} ) \otimes s_{42}
+ ( s_{1} + s_{2} + s_{3} + s_{4} ) \otimes s_{411}
+ ( s_{21} + s_{2} + s_{4} ) \otimes s_{33}\\
         &\quad   
+( s_{22} + s_{11} + 2s_{21} + 2s_{31} + s_{41} + s_{2} + 2s_{3} + 2s_{4} + 2s_{5} + s_{6} ) \otimes s_{321}\\
         &\quad   
+( s_{32} + s_{11} + s_{21} + 2s_{31} + s_{41} + s_{51} + s_{3} +  s_{4} + 2s_{5} + s_{6} + s_{7} ) \otimes s_{3111}\\
         &\quad   
+( s_{211} + s_{32} + s_{21} + s_{31} + s_{41} + s_{51} + s_{4} + s_{5} + s_{7} ) \otimes s_{222}\\
         &\quad   
+( s_{111} + s_{211} + s_{311} + s_{22} + s_{32} + s_{42} + s_{21} + 2s_{31} + 3s_{41} + 2s_{51}+ s_{61}   
         + s_{4} + s_{5} + 2s_{6} + s_{7} + s_{8} ) \otimes s_{2211}\\
         &\quad   
+( s_{111} + s_{211} + s_{311} + s_{411} + s_{33} + s_{32} + s_{42} + s_{52} 
	\\  &\quad\qquad\qquad\qquad
	 + s_{31} + 2s_{41} + 2s_{51} + 2s_{61} + s_{71} + s_{6} + s_{7} + s_{8} + s_{9} ) \otimes s_{21111}\\
         &\quad   
+( s_{1111} + s_{311} + s_{411} + s_{511} + s_{43} + s_{42} + s_{62} 
	+ s_{61} + s_{71} + s_{81} + s_{10.} ) \otimes s_{111111}
\end{aligned}$
\vfill 
\item[] $\begin{aligned}  
\Super_{51}&= s_{1} \otimes s_{51}
+ ( s_{11} + s_{21} + s_{2} + s_{3} + s_{4} ) \otimes s_{42}
+ ( s_{11} + s_{21} + s_{31} + s_{2} + s_{3} + s_{4} + s_{5} ) \otimes s_{411}\\
         &\quad
+ ( s_{21} + s_{31} + s_{3} + s_{5} ) \otimes s_{33}\\
         &\quad
+ ( s_{111} + s_{211} + 2s_{22} + s_{32} + 2s_{21} + 4s_{31} + 3s_{41} + s_{51} 
      + s_{3} + 2s_{4} + 2s_{5} + 2s_{6} + s_{7} ) \otimes s_{321}\\
         &\quad
+ ( s_{111} + s_{211} + s_{311} + s_{21} + s_{22} + 2s_{32} + s_{42} 
       + 2s_{31} + 3s_{41} + 2s_{51} + s_{61} \\
         &\quad\qquad\qquad\qquad
         	+ s_{4} + s_{5} + 2s_{6} + s_{7} + s_{8} ) \otimes s_{3111}\\
         &\quad
+ ( s_{211} + s_{311} + s_{22} + s_{32} + s_{42} + s_{31} + 2s_{41} + s_{51} + s_{61}
       + s_{5} + s_{6} + s_{8} ) \otimes s_{222}\\
         &\quad
+ ( s_{221} + 2s_{211} + 2s_{311} + s_{411} + s_{33} + s_{22} + 3s_{32} + 2s_{42} + s_{52}
       + s_{31} + 3s_{41}\\
         &\quad\qquad\qquad 
         + 4s_{51} + 2s_{61} + s_{71}
  + s_{5} + s_{6} + 2s_{7} + s_{8} + s_{9} ) \otimes s_{2211}\\
         &\quad
+ ( s_{1111} + s_{221} + s_{321} + s_{211} + 2s_{311} + 2s_{411} + s_{511} + s_{33} + s_{43}
         + s_{32} + 3s_{42}
       \\ &\quad\qquad\qquad  
              + 2s_{52} + s_{62} 
        + s_{41} + 2s_{51} + 3s_{61} + 2s_{71} + s_{81}
         + s_{7} + s_{8} + s_{9} + s_{(10)} ) \otimes s_{21111}\\
         &\quad
+ ( s_{2111} + s_{321} + s_{421} + s_{411} + s_{511} + s_{611} + s_{43} + s_{53} + s_{52} + s_{62}\\
         &\quad\qquad\qquad  + s_{71} + s_{72} + s_{81} + s_{91} + s_{(11)} ) \otimes s_{111111} 
\end{aligned}$
\item[] $\begin{aligned}  
\Super_{411}&= s_{2}  \otimes s_{51}
+ ( s_{21} + s_{31} + s_{3} + s_{4} + s_{5} ) \otimes s_{42}\\
         &\quad
+ ( s_{22} + s_{21} + s_{31} + s_{41} + s_{3} + s_{4} + s_{5} + s_{6} ) \otimes s_{411}\\
         &\quad
+ ( s_{211} + s_{22} + s_{31} + s_{41} + s_{4} + s_{6} ) \otimes s_{33}\\
         &\quad
+ ( 2s_{211} + s_{311} + s_{22} + 3s_{32} + s_{42} + 2s_{31} + 4s_{41} + 3s_{51} + s_{61}
         + s_{4} + 2s_{5} + 2s_{6} + 2s_{7} + s_{8} ) \otimes s_{321}\\
         &\quad
+ ( s_{221} + s_{211} + s_{311} + s_{411} + s_{33} + 2s_{32} + 2s_{42} + s_{52} 
        + s_{31} + 2s_{41}  
        \\  &\quad\qquad\qquad
        + 3s_{51} + 2s_{61} + s_{71} 
          + s_{5} + s_{6} + 2s_{7} + s_{8} + s_{9} ) \otimes s_{3111}\\
         &\quad
+ ( s_{221} + s_{311} + s_{411} + s_{33} + s_{32} + s_{42} + s_{52}
         + s_{41} + 2s_{51} + s_{61} + s_{71} + s_{6} + s_{7} + s_{9} ) \otimes s_{222}\\
         &\quad
+ ( s_{1111} + s_{221} + s_{321} + 3s_{311} + 2s_{411} + s_{511} 
       + s_{33} + s_{43}+ s_{32} + 4s_{42} + 2s_{52} + s_{62}\\
         &\quad\qquad\qquad+ s_{41} + 3s_{51} + 4s_{61} + 2s_{71} + s_{81}
         + s_{6} + s_{7} + 2s_{8} + s_{9} + s_{(10)} ) \otimes s_{2211}\\
         &\quad
+ ( s_{2111} + 2s_{321} + s_{421} + s_{311} + 2s_{411} + 2s_{511} + s_{611} 
        + s_{33} + 2s_{43} + s_{53} 
        \\ &\quad\qquad\qquad 
        + s_{42} + 3s_{52} + 2s_{62} + s_{72}
        + s_{51} + 2s_{61} + 3s_{71} + 2s_{81} + s_{91}
          + s_{8} + s_{9} + s_{(10)} + s_{(11)} ) \otimes s_{21111} \\
         &\quad
+ (s_{3111} + s_{331} + s_{421} + s_{521} + s_{511} + s_{611} + s_{711} 
        + s_{44} + s_{53} + s_{63} 
        \\ &\quad\qquad\qquad 
         + s_{62} + s_{72} + s_{82}
        + s_{81} + s_{91} + s_{(10,1)} + s_{(12)})\otimes s_{111111}
\end{aligned}$

\item[] $\begin{aligned}  
\Super_{3111}&=s_{3}  \otimes s_{51}
+ ( s_{22} + s_{31} + s_{41} + s_{4} + s_{5} + s_{6} ) \otimes s_{42}\\
         &\quad
+ ( s_{32} + s_{31} + s_{41} + s_{51} + s_{4} + s_{5} + s_{6} + s_{7} ) \otimes s_{411}
+ ( s_{32} + s_{41}  + s_{51} + s_{5} + s_{7} ) \otimes s_{33}\\
         &\quad
+ ( s_{221} + 2s_{311} + s_{411} + s_{33} + 2s_{32} + 3s_{42} + s_{52} 
         + 2s_{41} + 4s_{51}
       \\ &\quad\qquad\qquad  
          + 3s_{61} + s_{71} 
       + s_{5} + 2s_{6} + 2s_{7} + 2s_{8} + s_{9} ) \otimes s_{321}\\
         &\quad
+ ( s_{321} + s_{311} + s_{411} + s_{511} + s_{33} + s_{43} + 2s_{42} + 2s_{52} + s_{62} \\
         &\quad\qquad\qquad  + s_{41} + 2s_{51}  + 3  s_{61} + 2s_{71} + s_{81} 
         + s_{6} + s_{7} + 2s_{8} + s_{9} + s_{(10)} ) \otimes s_{3111}\\
         &\quad
+ ( s_{321} + s_{311} + s_{411} + s_{511} + s_{43} + 2s_{42} + s_{52} + s_{62} 
       \\ &\quad\qquad\qquad  
         + s_{51} + 2s_{61} + s_{71} + s_{81} + s_{7} + s_{8} + s_{(10)} ) \otimes s_{222}\\
         &\quad
+ ( s_{2111} + 2s_{321} + s_{421} + 3s_{411} + 2s_{511} + s_{611} + s_{33} + 2s_{43} + s_{53} 
        + s_{42} + 4s_{52} + 2s_{62} + s_{72}
         \\ &\quad\qquad\qquad   
         + s_{51} + 3s_{61} + 4s_{71} + 2s_{81} + s_{91}
         + s_{7} + s_{8} + 2s_{9} + s_{(10)} + s_{11.} ) \otimes s_{2211}\\
         &\quad
+ ( s_{3111} + s_{331} + 2s_{421} + s_{521} + s_{411} + 2s_{511} + 2s_{611} + s_{711}
         + s_{44} + s_{43} + 2s_{53} + s_{63}
         \\ &\quad\qquad\qquad 
         + s_{52} + 3s_{62} + 2s_{72} + s_{82}
         + s_{61} + 2s_{71} + 3s_{81} + 2s_{91} + s_{(10,1)}
         + s_{9} + s_{(10)} + s_{(11)} + s_{(12)} ) \otimes s_{21111} \\
         &\quad
 +(s_{4111} + s_{431} + s_{521} + s_{621} + s_{611} + s_{711} + s_{811} 
	+ s_{54} + s_{63} + s_{73} 
	+ s_{72} + s_{82}  
	\\ &\quad\qquad\qquad
	+ s_{91} + s_{92} + s_{(10,1)} + s_{(11,1)} + s_{(13)})\otimes s_{111111}
\end{aligned}$

\vfill 

\item[] $\begin{aligned}  
\Super_{21111}&= s_{4}  \otimes s_{51}
+ ( + s_{32} s_{31} + s_{41} + s_{51} + s_{5} + s_{6} + s_{7} ) \otimes s_{42}\\
         &\quad
+ ( s_{32} + s_{42} + s_{41} + s_{51} + s_{61} + s_{5} + s_{6} + s_{7} + s_{8} ) \otimes s_{411}\\
         &\quad
+ ( s_{311} + s_{22} + s_{42} + s_{41} + s_{51} + s_{61} + s_{6} + s_{8} ) \otimes s_{33}\\
         &\quad
+ ( s_{221} + s_{321} + s_{311} + 2s_{411} + s_{511} + s_{33} + s_{43}
         + s_{32} + 3s_{42} + 3s_{52} + s_{62}
         \\ &\quad\qquad\qquad 
         + s_{41} + 3s_{51} + 4s_{61} + 3s_{71}+ s_{81}
         + s_{6} + 2s_{7} + 2s_{8} + 2s_{9} + s_{(10)} ) \otimes s_{321}\\
         &\quad
+ ( s_{321} + s_{421} + s_{411} + s_{511} + s_{611} + s_{33} + s_{43} + s_{53}
	+ s_{42} + 3s_{52} + 2s_{62} + s_{72}
         \\ &\quad\qquad\qquad 
         + s_{51} + 2s_{61} + 3s_{71} + 2s_{81} + s_{91}
         + s_{7} + s_{8} + 2s_{9} + s_{(10)} + s_{(11)} ) \otimes s_{3111}\\
         &\quad
+ ( s_{2111} + s_{321} + s_{421} + s_{33} + s_{411} + s_{511} + s_{611}
	+ s_{43} + s_{53} + 2s_{52} + s_{62} + s_{72}
	\\ &\quad\qquad\qquad 
	+ s_{51} + s_{61} + 2s_{71} + s_{81} + s_{91} + s_{8} + s_{9} + s_{(11)} ) \otimes s_{222}\\
         &\quad
+ ( s_{3111} + s_{331} + s_{321} + 2s_{421} + s_{521} + s_{411} + 3s_{511} + 2s_{611} + s_{711} \\
         &\quad\qquad\qquad + s_{44} + 2s_{43} + 2s_{53} + s_{63}
         + s_{42} + 2s_{52} + 4s_{62} + 2s_{72} + s_{82}\\
         &\quad\qquad\qquad + 2s_{61} + 3s_{71} +  4s_{81} + 2s_{91} + s_{(10,1)}
         + s_{8} + s_{9} + 2s_{(10)} + s_{(11)} + s_{(12)} ) \otimes s_{2211}\\
         &\quad
+ ( s_{4111} + s_{331} + s_{431} + s_{421} + 2s_{521} + s_{621} 
      + s_{511} + 2s_{611} + 2s_{711} + s_{811}\\
         &\quad\qquad\qquad  + s_{54} + s_{43} + 2s_{53} + 2s_{63} + s_{73}
         + 2s_{62} + 3s_{72} + 2s_{82} + s_{92}\\
         &\quad\qquad\qquad  + s_{71} + 2s_{81} + 3s_{91} + 2s_{(10,1)} + s_{(11,1)}
         + s_{(10)} + s_{(11)} + s_{(12)} + s_{(13)}) \otimes s_{21111}\\
         &\quad
+ (s_{5111} + s_{431} + s_{531} + s_{621} + s_{721} + s_{711} + s_{811} + s_{911} 
	+ s_{64} + s_{63} + s_{73} + s_{83}  
	\\ &\quad\qquad\qquad  
	+ s_{82} + s_{92}+ s_{(10,2)} 
	+ s_{(10,1)} + s_{(11,1)} + s_{(12,1)} + s_{(14)})\otimes s_{111111}
 \end{aligned}$

\end{itemize}

  (The value of $\Super_{111111}=\E_6$ may be found in~\cite{MRC1}.)


\subsection*{The \texorpdfstring{$e$}{e}-expansions of the \texorpdfstring{$\F_{\rho}$}{fmn}'s for hooks of size \texorpdfstring{$\leq 4$}{4}.}
\ \null\medskip

\begin{itemize}\setlength\itemsep{4pt}\setlength{\itemindent}{-.5in}
\item[]$ \F_{1}=1\otimes  e_{1};$
		\smallskip
\item[]$\F_{2}= 1\otimes e_{2},$
\item[]$\F_{11}= 1\otimes e_{11} + s_{1} \otimes e_{2};$
		\smallskip
\item[]$\F_{3}=1\otimes e_{21} + s_{1} \otimes e_{3},$
\item[]$\F_{21}=1\otimes e_{21} + s_{1} \otimes e_{21} + s_{2} \otimes e_{3},$
\item[]$\F_{111}=1\otimes e_{111} + (2s_{1} + s_{2}) \otimes e_{21} + (s_{11}+s_{3}) \otimes e_{3};$
		\smallskip
\item[]$\F_{4}= 1 \otimes e_{211} + s_{1} \otimes e_{22} + (s_{1}+ s_{2})\otimes e_{31} + (s_{11} + s_{3}) \otimes e_{4},$
\item[]$\F_{31}= (1  + s_{1} )\otimes e_{211} + s_{2} \otimes e_{22} + (s_{11} + s_{1} + s_{2} + s_{3}) \otimes e_{31} + (s_{21} + s_{4})\otimes e_{4},$
\item[]$\begin{aligned}
         &\F_{211}= (1   + s_{1}+ s_{2}) \otimes e_{211} + (s_{11} + s_{1} + s_{3})\otimes e_{22} 
	+ (s_{21} + s_{2}  + s_{3} + s_{4}) \otimes e_{31}+( s_{31}  + s_{5} )\otimes e_{4},
	\end{aligned}$
\item[]$\begin{aligned}
	&\F_{1111}= 1  \otimes e_{1111} + (3s_{1} + 2s_{2} + s_{3}) \otimes e_{211} + (s_{11} + s_{21} + s_{2}  + s_{4}) \otimes e_{22} \\
	&\quad \qquad 
	+ (2s_{11} + s_{21} + s_{31} + 2s_{3} + s_{4} + s_{5}) \otimes e_{31}+ (s_{111}  + s_{31} + s_{41}  + s_{6}) \otimes e_{4}.
	\end{aligned}$
\end{itemize}
}

\renewcommand{\refname}{\bleu{References}}

\end{document}